\documentclass[11pt,leqno]{amsart}
\usepackage{verbatim,amssymb,amsfonts,latexsym,amsmath,amsthm,tikz,bbm,nicefrac,xspace,graphicx,mathtools,multicol}

\newtheorem{theorem}{Theorem}[section]

\newtheorem*{maintheorem}{Main Theorem}
\newtheorem*{theorem*}{Theorem}
\newtheorem*{corollary*}{Corollary}

\newtheorem*{sub-claim}{sub-claim}
\theoremstyle{definition}
\newtheorem{example}[theorem]{Example}
\newtheorem{definition}[theorem]{Definition}

\theoremstyle{remark}

\newtheorem*{definition*}{Definition}

\newcommand{\R}{\mathbb{R}}
\newcommand{\N}{\mathbb{N}}

\newcommand{\C}{\mathcal{C}}

\newcommand{\explicitSet}[1]{\left\lbrace #1 \right\rbrace}
\newcommand{\brackets}[1]{\left\langle #1 \right\rangle}
\newcommand{\set}[2]{\explicitSet{#1 \colon #2}}
\newcommand{\seq}[2]{\brackets{#1 \colon #2}}

\newcommand{\g}{\gamma}
\newcommand{\dlt}{\delta}
\newcommand{\e}{\varepsilon}

\newcommand{\0}{\emptyset}
\newcommand{\sub}{\subseteq}
\newcommand{\rest}{\!\restriction\!}
\newcommand{\setmins}{\hspace{-.25mm} \setminus \hspace{-.25mm}}

\newcommand{\norm}[1]{\lVert#1\rVert}

\newcommand{\T}{\mathcal{T}}

\begin{document}

\title{Linear operators with infinite entropy}
\author[W. R. Brian]{Will Brian}
\address[W. R. Brian]{
Department of Mathematics and Statistics,
University of North Carolina at Charlotte,
9201 University City Blvd.,
Charlotte, NC 28223}
\email{wbrian@uncc.edu}
\author[J. P. Kelly]{James P. Kelly}
\address[J. P. Kelly]{Department of Mathematics, Christopher Newport University, Newport News, VA 23606, USA}
\email{james.kelly@cnu.edu}
\subjclass[2010]{54H20, 47A16, 47B37, 37B40}
\keywords{translation operators, weighted Lebesgue space, topological entropy}

\begin{abstract}
We examine the chaotic behavior of certain continuous linear operators on infinite-dimensional Banach spaces, and provide several equivalent characterizations of when these operators have infinite topological entropy. 


For example, it is shown that infinite topological entropy is equivalent to non-zero topological entropy for translation operators on weighted Lebesgue function spaces. In particular, finite non-zero entropy is impossible for this class of operators, which answers a question raised by Yin and Wei.
\end{abstract}

\maketitle

\section{Introduction}

In this paper we examine the chaotic behavior of certain continuous linear operators on infinite-dimensional Banach spaces: namely, left translation operators on the weighted Lebesgue function spaces $L^p_v(\R_+)$ and on the related spaces $C_{0,v}(\R_+)$. Our main result is to characterize when such operators have infinite topological entropy.

The notion of ``chaotic behavior'' does not have a single precise meaning in topological dynamics: it indicates vaguely that a dynamical system becomes mixed up and disordered over time. This vague notion has been made precise via many different (and inequivalent) definitions of chaos: there is Devaney chaos \cite{Devaney}, or the specification property \cite{Bowen,KLO}, or the property of having infinite topological entropy \cite{AKM}, for example. Similarly, there are many different (and inequivalent) notions of anti-chaotic behavior in topological dynamical systems: equicontinuity, for example, or the property of having zero topological entropy. 

Continuous linear mappings on infinite-dimensional vector spaces can be highly chaotic; see, e.g., \cite{Bayart&Matheron,GEP}.
However, for such mappings, and in particular for the translation operators discussed in this paper, the distinctions between many different notions of chaos and anti-chaos disappear. Strong forms of chaos become equivalent to seemingly much weaker forms of chaos, or to the mere absence of certain anti-chaos properties.

However, not all forms of chaotic behavior are equivalent for the operators under discussion. The picture that emerges in this paper seems to be that the translation operators on $L^p_v(\R_+)$ organize themselves, for the most part, into three tiers of increasingly chaotic behavior. (See Figure 1 in the following section.) 
Our main result is to delineate the weakest tier of chaotic behavior by proving that several weak versions of chaos are equivalent in this context:
\begin{maintheorem}
Let $X$ be one of the Banach spaces $L^p_v(\R_+)$ or $C_{0,v}(\R_+)$, where $v$ is an admissible weight function, and let $\T = \set{T_t}{t \in \R_+}$ be the semigroup of left translation operators on $X$. The following are equivalent:
\begin{enumerate}
\item $\sup \!\set{ \frac{v(x)}{v(y)} }{ y \geq x } = \infty$.
\item For some $f \in X$, $\lim_{t \to \infty} T_tf \,\neq\, 0$.
\item $\T$ is not uniformly bounded.
\item $\T$ is not uniformly equicontinuous.
\item $\T$ is not equicontinuous.
\item $\T$ has nonzero entropy.
\item $\T$ has infinite entropy.
\end{enumerate}
\end{maintheorem}

\noindent All of the terms in the statement of this theorem will be defined in the following section. The equivalence of $(6)$ and $(7)$ shows that finite, nonzero entropy is impossible for the left translation operators on the weighted Lebesgue function spaces $L^p_v(\R_+)$, answering a question of Yin and Wei \cite{Yin&Wei}.

In addition to this theorem, we also prove that a close relative of $(2)$ is equivalent to hypercyclicity, and thus fits into the middle tier of chaos mentioned above. Namely, we show that $\T$ is hypercyclic (i.e., has a dense orbit) if and only if there is some $f \in X$ and some $[0,a] \sub \R_+$ such that $\lim_{t \to \infty} \norm{(T_tf) \chi_{{}_{[0,a]}}} \,\neq\, 0$. In other words, $\T$ is hypercyclic if and only if some $f \in X$ does not tend to $0$ on some fixed bounded interval.

\section{Preliminaries}

Let $\R_+ = [0,\infty)$. The \emph{weighted Lebesgue space} $L^p_v(\R_+)$ is defined as
$$L^p_v(\R_+) = \left\{ f \colon \R_+ \to \R \ : \int_0^\infty |f(x)|^p v(x) \,dx < \infty \right\},$$
where $0 < p < \infty$ and $v: \R_+ \to \R_+$ is an \emph{admissible weight function},
which means that  $v$ is strictly positive, locally integrable, and there exist some $M \geq 1$ and $w\in\R_+$ such that 
$$v(x) \leq Me^{wt}v(x+t)$$
for all $t \geq 0$. This definition follows \cite[chapter 7]{GEP}. 

The admissibility condition ensures that the translation operators $T_t$ (defined below) are all continuous, and that they form a strongly continuous semigroup under composition. 
It is also worth noting that for admissible $v(x)$, on any finite interval $[0,a] \sub \R_+$ we have $\inf_{x \in [0,a]} v(x) \geq \nicefrac{v(0)}{Me^{wa}} > 0$, and $\sup_{x \in [0,a]} v(x) \leq Me^{wa}v(a) < \infty$. Thus the condition on $v(x)$ in the statement of our main theorem, that $\sup \!\set{ \frac{v(x)}{v(y)} }{ y \geq x } = \infty$, cannot be satisfied on a bounded interval; in other words, it is a condition concerning the behavior of $v(x)$ ``at infinity.''

Formally, we consider two functions in $L^p_v(\R_+)$ to be equal if they are equal Lebesgue almost everywhere. This does not play an important part in what follows, and we will abuse notation slightly (as above) by speaking of the elements of $L^p_v$ as functions, and not equivalence classes of functions. However, making this identification allows us to assert that setting
$$\norm{f} \,=\, \left( \int_0^\infty |f(x)|^pv(x) \,dx \right)^{\!\nicefrac{1}{p}}$$
defines a norm on $L^p_v(\R_+)$.
With this norm, $L^p_v(\R_+)$ is a Banach space.

Similarly, we also define
$$C_{0,v}(\R_+) = \left\{ f \colon \R_+ \to \R \ \colon f \text{ is continuous, and} \lim_{x\to\infty} |f(x)| v(x) = 0 \right\}.$$
We define a norm on $C_{0,v}(\R_+)$ by
$$\norm{f} \,=\, \sup\set{ |f(x)|v(x) }{x \in \R_+},$$
and note that $C_{0,v}$ is a Banach space with this norm.

Let $X$ denote one of the Banach spaces $L^p_v(\R_+)$ or $C_{0,v}$, where $v$ is an admissible weight function. For each $t \geq 0$, let $T_t: X \to X$ denote the left translation operator defined by setting
$$T_t(f)(x) = f(x+t)$$
for all $x \geq 0$. Let $\T = \set{T_t}{t \geq 0}$, and note that each member of $\T$ is a continuous linear operator on $X$.

The \emph{orbit} of a point $f \in X$ under $\T$ is $\set{ T_tf }{ t \geq 0 } \sub X$. We say $\T$ is \emph{hypercyclic} if there is a point whose orbit is a dense subset of $X$. A point $f \in X$ is \emph{periodic} for $\T$ if there exists $t > 0$ such that $T_t f = f$. We say that $\T$ is \emph{Devaney chaotic} if it is hypercyclic and the set of periodic points is dense in $X$.

The following theorem delineates the ``strongest tier of chaos" described in the introduction for the spaces $L^p_v(\R_+)$ and $C_{0,v}(\R_+)$. A proof can be found in \cite{BKT}, although many of the implications inherent in the theorem predate that paper (see the references therein).
We refer the reader to \cite{BMGMAP} or \cite{Mangino&Peris} for the definition of \emph{frequently hypercyclic}; roughly, it states that some member of $X$ not only has a dense orbit, but that it visits each open subset of $X$ ``frequently.'' We refer the reader to \cite{BMGP} or \cite{BKT} for the definition of the specification property in the present context, an adaptation of Bowen's definition for compact metric spaces in \cite{Bowen}. 

\noindent\begin{minipage}{\textwidth}
\begin{theorem*}
Let $X = L^p_v(\R_+)$, where $v$ is an admissible weight function, and let $\T = \set{T_t}{t \in \R_+}$ be the semigroup of left translation operators on $X$. The following are equivalent:
\begin{enumerate}
\item $\displaystyle \int_0^\infty v(x) \,dx < \infty$.
\vspace{1mm}
\item $\T$ is frequently hypercyclic.
\item $\T$ has the specification property.
\item $\T$ has Devaney chaos.
\item Some $f \in X \setmins \{0\}$ is periodic.
\end{enumerate}
\end{theorem*}
\vspace{1mm}
\end{minipage}

In \cite{BKT}, it was shown that these equivalent conditions all imply that $\T$ has infinite entropy, but that this implication does not reverse. It is worth mentioning that some of these results hold in a broader context than left translation operators on $L^p_v(\R_+)$. For example, the specification property is equivalent to Devaney chaos for backward shift operators on Banach sequence spaces \cite{BMGP} and for weighted backward shifts on sequence $F$-spaces \cite{BMGMAP}. However, the above theorem does not hold with $C_{0,v}(\R_+)$ in the place of $L^p_v(\R_+)$.

Recall that a strongly continuous semigroup $\T=\set{T_t}{t\geq0}$ on a space $X$ is \emph{topologically transitive} if for all nonempty open $U,V \sub X$, $T_t(U) \cap V \neq \0$ for arbitrarily large $t$. If $X$ is separable, then this is equivalent to $\T$ being hypercyclic (having a point with a dense orbit). In \cite{DSW}, it was shown that the translation operators $T_t$ on $L^p_v(\R_+)$ or $C_{0,v}(\R_+)$ are topologically transitive (and thus hypercyclic) if and only if $\liminf_{x \to \infty} v(x) = 0$. Of course, this condition on $v(x)$ is strictly weaker than the integrability condition from the previous theorem.
In \cite{Yin&Wei}, Yin and Wei show a number of other chaotic behaviors to be equivalent to the hypercyclicity of the mapping semigroup $\T$. These results are summarized in the following theorem.

\begin{theorem*}
Let $X$ be one of the Banach spaces $L^p_v(\R_+)$ or $C_{0,v}(\R_+)$, where $v$ is an admissible weight function, and let $\T = \set{T_t}{t \in \R_+}$ be the semigroup of left translation operators on $X$. The following are equivalent:
\begin{enumerate}
\item $\displaystyle \liminf_{x \to \infty} v(x) = 0$.
\vspace{1mm}
\item $\T$ is hypercyclic.
\item $\T$ is topologically transitive.
\item Some $f \in X \setmins \{0\}$ is a recurrent point of $\T$.
\item Some $f \in X \setmins \{0\}$ is a non-wandering point of $\T$.
\item Some $f \in X \setmins \{0\}$ has a non-trivial $\omega$-limit set.
\end{enumerate}
\end{theorem*}

\noindent Here ``trivial'' means either empty or equal to $\{0\}$. We prove in the following section that equivalent to all these conditions is

\begin{itemize}
\item[$(7)$] There is some $f \in X$ and some bounded $[0,a] \sub \R_+$ such that $$\lim_{t \to \infty} \norm{(T_tf) \chi_{{}_{[0,a]}}} \neq 0.$$
\end{itemize}

\noindent Together, these seven equivalent statements represent the ``middle tier'' of chaos mentioned in the introduction.

Yin and Wei also show in \cite{Yin&Wei} that:
\begin{itemize}
\item[$\circ$] Any of the conditions from the previous theorem(s) imply that $\T$ has infinite entropy.
\item[$\circ$] If the weight function $v(x)$ is bounded, then each of the conditions from the previous theorem is equivalent to $\T$ having infinite entropy.
\item[$\circ$] There is an unbounded weight function $v(x)$ such that the semigroup of translation operators on $L^p_v(\R_+)$ is not hypercyclic, but nonetheless has infinite entropy.
\end{itemize}

The following picture summarizes the results of the previous two theorems, along with the main result of this paper stated in the introduction.

\vspace{2mm}
\begin{center}
\begin{figure}[h]
\begin{tikzpicture}[xscale=1,yscale=.9]

\draw[rounded corners] (-6,9) rectangle (6,12.25);

\node [anchor=center] at (-2.3,10.4) {\footnotesize $\T$ has}; 
\node [anchor=center] at (-2.3,10) {\footnotesize Devaney}; 
\node [anchor=center] at (-2.3,9.6) {\footnotesize chaos}; 

\node [anchor=center] at (-4.6,11.65) {\footnotesize $\T$ has the}; 
\node [anchor=center] at (-4.6,11.25) {\footnotesize specification}; 
\node [anchor=center] at (-4.6,10.85) {\footnotesize property}; 

\node [anchor=center] at (0,11.3) {\footnotesize $\displaystyle \int_0^\infty \!\!\!v(x) \,dx \,<\, \infty$}; 

\node [anchor=center] at (2.3,10.4) {\footnotesize $\T$ admits a}; 
\node [anchor=center] at (2.3,10.0) {\footnotesize nonzero}; 
\node [anchor=center] at (2.3,9.6) {\footnotesize periodic point}; 

\node [anchor=center] at (4.6,11.65) {\footnotesize $\T$ is}; 
\node [anchor=center] at (4.6,11.25) {\footnotesize frequently}; 
\node [anchor=center] at (4.6,10.85) {\footnotesize hypercyclic}; 

\draw[rounded corners] (-6,3) rectangle (6,7.25);

\node [anchor=center] at (-2.3,6.45) {\footnotesize $\T$ is transitive}; 

\node [anchor=center] at (2.3,6.45) {\footnotesize $\T$ is hypercyclic}; 

\node [anchor=center] at (-4.6,5.65) {\footnotesize there is a}; 
\node [anchor=center] at (-4.6,5.25) {\footnotesize nontrivial}; 
\node [anchor=center] at (-4.6,4.85) {\footnotesize $\omega$-limit set}; 

\node [anchor=center] at (-2.3,4.4) {\footnotesize $\T$ admits a}; 
\node [anchor=center] at (-2.3,4) {\footnotesize nonzero}; 
\node [anchor=center] at (-2.3,3.6) {\footnotesize nonwandering point}; 

\node [anchor=center] at (0,5.3) {\footnotesize $\displaystyle \liminf_{x \to \infty} \, v(x) \,=\, 0$}; 

\node [anchor=center] at (2.3,4.4) {\footnotesize $\T$ admits a}; 
\node [anchor=center] at (2.3,4.0) {\footnotesize nonzero}; 
\node [anchor=center] at (2.3,3.6) {\footnotesize recurrent point}; 

\node [anchor=center] at (4.6,5.7) {\footnotesize on some $[0,a]$,}; 
\node [anchor=center] at (4.6,5.3) {\footnotesize not every point}; 
\node [anchor=center] at (4.6,4.9) {\footnotesize tends to $0$}; 

\draw[rounded corners] (-6,-2.3) rectangle (6,1.25);

\node [anchor=center] at (-4.2,.45) {\footnotesize $\T$ has}; 
\node [anchor=center] at (-4.2,.05) {\footnotesize infinite entropy}; 

\node [anchor=center] at (-4.2,-1.1) {\footnotesize $\T$ has}; 
\node [anchor=center] at (-4.2,-1.5) {\footnotesize nonzero entropy}; 

\node [anchor=center] at (0,.3) {\footnotesize $\displaystyle \sup \set{\textstyle \frac{v(x)}{v(y)}}{y \geq x} = \infty$}; 

\node [anchor=center] at (4.2,.45) {\footnotesize not every point}; 
\node [anchor=center] at (4.2,.05) {\footnotesize tends to $0$}; 

\node [anchor=center] at (4.2,-.9) {\footnotesize $\T$ is not}; 
\node [anchor=center] at (4.2,-1.3) {\footnotesize (uniformly)}; 
\node [anchor=center] at (4.2,-1.7) {\footnotesize equicontinuous}; 

\node [anchor=center] at (0,-.9) {\footnotesize $\T$ is not}; 
\node [anchor=center] at (0,-1.3) {\footnotesize uniformly}; 
\node [anchor=center] at (0,-1.7) {\footnotesize bounded}; 

\draw (-.3,8.7)--(-.3,7.77);
\draw (.3,8.7)--(.3,7.77);
\draw (-.45,7.9)--(0,7.5);
\draw (.45,7.9)--(0,7.5);

\draw (-.3,2.7)--(-.3,1.77);
\draw (.3,2.7)--(.3,1.77);
\draw (-.45,1.9)--(0,1.5);
\draw (.45,1.9)--(0,1.5);

\end{tikzpicture}
\caption{Three main tiers of chaotic behavior for $L^p_v(\R_+)$}
\end{figure}
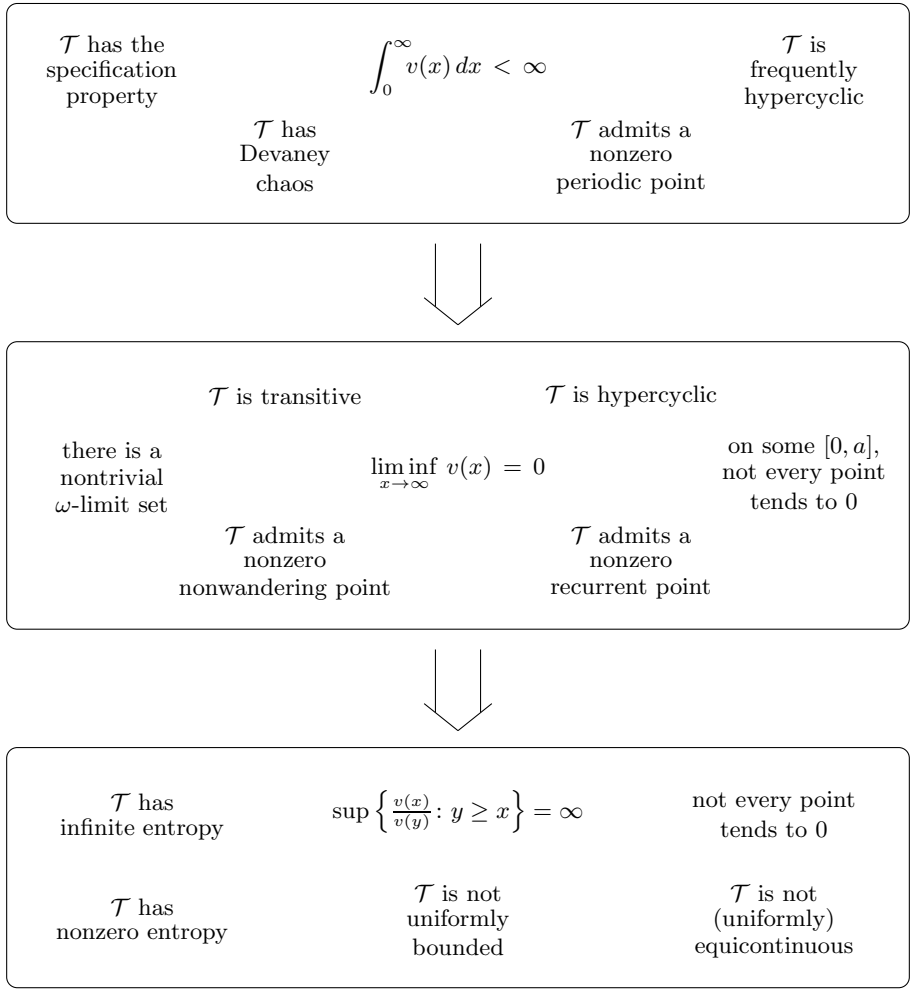
\end{center}

As this picture indicates, surprisingly many of the (ostensibly different) possible behaviors of $L^p_v(\R_+)$ fall into only three categories. Many, but not all. It is important to remember that this picture is incomplete, and is not meant to suggest that the $L^p_v(\R_+)$ can exhibit only three topologically distinct types of behaviors. This is discussed further in Section 4.

\vspace{2mm}

The notion of topological entropy was introduced by Adler, Konheim, and McAndrew in \cite{AKM}. We use the notation developed by Bowen in \cite{Bowen}.
	
\begin{definition}
Let $X$ be a separable space with translation-invariant metric $d$, and let $\T = \set{T_t}{t \geq 0}$ be a strongly continuous semigroup of operators on $X$. Let $K \sub X$ be compact, and let $t > 0$, and let $\e > 0$. A set $S \sub K$ is called \emph{$(t,\e)$-separated} if for every $f,g \in S$ with $f \neq g$, we have $d(T_uf,T_ug) \geq \e$ for some $u \in [0,t]$. We denote by $s_{t,\e}(\T,K)$ the largest cardinality of a $(t,\e)$-separated subset of $K$, and define
$$h(\T,K) \,=\, \lim_{\e \to 0} \, \limsup_{t \to \infty} \, \frac{1}{t} \log s_{t,\e}(\T,K)$$

If $X$ is already compact, then the topological entropy of $\T$ is defined as $h(\T,X)$. For non-compact spaces (such as the weighted Lebesgue spaces discussed here), the \emph{topological entropy} of $\T$ is
$$h(\T) \,=\, \sup \set{ h(\T,K) }{ K \sub X \text{ is compact} }.$$
\end{definition}

Let us note that the topological entropy of a single operator $T: X \to X$ can be expressed in our notation as the topological entropy of the semigroup $\T = \set{T^n}{n \in \N}$. In the case of continuous semigroups of operators, studying the full semigroup $\T = \set{T_t}{t \in \R_+}$ is related (and often equivalent) to studying a discretization of it. For example, for some fixed $t_0 > 0$, the entropy of the single operator $T_{t_0}$ is related to the entropy of the continuous semigroup $\T$ via the equation $h(T_{t_0}) = t_0h(\T)$ \cite{Hoock}.

A family $\T$ of functions on a metric space $X$ is \emph{equicontinuous} if for every $\e > 0$ and $x \in X$, there is some $\dlt > 0$ such that for any $T \in \T$ and any $y \in X$, if $d(x,y) < \dlt$ then $d(T(x),T(y)) < \e$. By not allowing $\dlt$ to depend on $x$, we arrive at a stronger notion: $\T$ is \emph{uniformly equicontinuous} if for every $\e > 0$ there is some $\dlt > 0$ such that for any $T \in \T$ and any $x,y \in X$, if $d(x,y) < \dlt$ then $d(T(x),T(y)) < \e$.

A family $\T$ of operators on a Banach space $X$ is \emph{uniformly bounded} if there is some $B \geq 0$ such that for any $x \in X$ and $T \in \T$, $\norm{T(x)} \leq B\norm{x}$.

\section{Two theorems}

We begin this section with our characterization of hypercyclicity stated in the previous section.

\begin{theorem}\label{thm:bonus}
Let $X$ be one of the Banach spaces $L^p_v(\R_+)$ or $C_{0,v}(\R_+)$, where $v$ is an admissible weight function, and let $\T$ be the semigroup of left translation operators on $X$. Then $\T$ is hypercyclic if and only if there is some $f \in X$ and some bounded $[0,a] \sub \R_+$ such that 
$\textstyle \lim_{x \to \infty} \norm{(T_tf) \chi_{{}_{[0,a]}}} \neq 0.$
\end{theorem}

\begin{proof}
Let us denote by $(*)$ the statement that there is some $f \in X$ and some $a > 0$ such that $\lim_{t \to \infty} \norm{(T_tf) \chi_{{}_{[0,a]}}} \neq 0$.

By the results stated in the previous section, it suffices to show that the existence of a nonzero recurrent point for $\T$ implies $(*)$, and that $(*)$ implies $\liminf_{x \to \infty} v(x) = 0$.

Let us show first that $(*)$ implies $\liminf_{x \to \infty} v(x) = 0$ by proving the contrapositive. Suppose $\liminf_{x \to \infty} v(x) \neq 0$. As $v$ is strictly positive, this means there exists $c>0$ such that $v(x)>c$ for sufficiently large $x$. Let $f \in X$ and for every $a > 0$, let $B_a = \sup \set{v(x)}{x \in [0,a]}$. Note that $B_a$ is well-defined (i.e., $\set{v(x)}{x \in [0,a]}$ is bounded above) because $v$ is admissible, which implies $v(x) \leq Me^{wa}v(a)$ for all $x \in [0,a]$. We consider two cases, according to whether $X = L^p_v(\R_+)$ or $X = C_{0,v}(\R_+)$.

For the first case, suppose $X = L^p_v(\R_+)$. Then
\begin{align*}
\norm{ (T_tf)\chi_{{}_{[0,a]}} }^p
& = \int_0^a |f(x+t)|^p \, v(x) \,dx \\
& = \int_0^a |f(x+t)|^p \, v(x+t) \frac{v(x)}{v(x+t)} \,dx \\
& \leq \frac{B_a}{c} \int_0^a |f(x+t)|^p \, v(x+t) \,dx \\
& \leq \frac{B_a}{c} \int_t^{t+a} |f(x)|^p \, v(x) \,dx
\end{align*}
and this goes to 0 as $t \to \infty$.

The second case is similar. Suppose $X=C_{0,v}(\R_+)$. Then
\begin{align*}
\norm{ (T_tf)\chi_{{}_{[0,a]}} }_\infty
& \,=\, \sup \set{ f(x+t)v(x) }{ x\in[0,a] } \\
& \,=\, \sup \set{ f(x+t)v(x+t)\frac{v(x)}{v(x+t)} }{ x \in [0,a] } \\
& \leq\, \frac{B_a}{c} \sup \set{ f(x+t)v(x+t) }{ x \in [0,a] } \\
& \leq\, \frac{B_a}{c} \sup \set{ f(x)v(x) }{ x \in [t,t+a] }
\end{align*}
which goes to 0 as $t \to \infty$. Note that this is really the same proof as in the first case, except that we must reinterpret an integral sign as a supremum.
In either case, $(*)$ implies $\liminf_{x \to \infty} v(x) = 0$. 

Next we show that the existence of a nonzero recurrent point for $\T$ implies $(*)$. The proofs for $X = L^p_v(\R_+)$ and $X = C_{0,v}(\R_+)$ are once again very similar. (Once again, the difference amounts to reinterpreting an integral sign as a supremum.) So we give the proof only for $X = L^p_v(\R_+)$.

Suppose that $f \in X \setmins \{0\}$ is recurrent. Then we may find an increasing sequence $t_1 < t_2 < t_3 < \dots $ of real numbers, with $\lim_{k \to \infty}t_k = \infty$, such that $T_{t_k}f \in B(f,\nicefrac{1}{k})$ for all $k \in \N$. 

As $\norm{f} = \lim_{a \to \infty} \norm{f \chi_{{}_{[0,a]}}}$, we may choose some $a > 0$ such that 
$$\norm{f \chi_{{}_{[0,a]}}} \,>\, \frac{1}{2} \norm{f}.$$

Note that for any $g,h \in X$, the distance from $g\chi_{{}_{[0,a]}}$ to $h\chi_{{}_{[0,a]}}$ is bounded by the distance from $g$ to $h$: 
\begin{align*}
\norm{g\chi_{{}_{[0,a]}}-h\chi_{{}_{[0,a]}}} &\,=\, \left( \int_0^a |g(x)-h(x)|^p \,dx \right)^{\!\nicefrac{1}{p}} \\
&\,\leq\, \left( \int_0^\infty |g(x)-h(x)|^p \,dx \right)^{\!\nicefrac{1}{p}} \,=\, \norm{g-h}.
\end{align*}
In particular, $(T_{t_k}f)\chi_{{}_{[0,a]}} \in B(f\chi_{{}_{[0,a]}},\nicefrac{1}{k})$ for all $k \in \N$.
By our choice of $a$, this implies
$$\norm{(T_{t_k}f)\chi_{{}_{[0,a]}}} \,\geq\, \norm{f\chi_{{}_{[0,a]}}} - \frac{1}{k} \,>\, \frac{1}{2}\norm{f} - \frac{1}{k}.$$
As $\norm{f} > 0$, this shows, as claimed, that $\norm{(T_tf)\chi_{{}_{[0,a]}}}$ does not converge to $0$ as $t \to \infty$.
\end{proof}

The first half of this proof works essentially because, when the weight function $v(x)$ is restricted to a bounded interval $[0,a]$, then $\sup_{x,y \in [0,a]}\frac{v(x)}{v(y)}$ is well-defined. Our condition on $v$ in the following theorem simply allows us to apply the same idea on unbounded intervals.

\begin{theorem}\label{thm:main}
Let $X$ be one of the Banach spaces $L^p_v(\R_+)$ or $C_{0,v}(\R_+)$, where $v$ is an admissible weight function, and let $\T = \set{T_t}{t \in \R_+}$ be the semigroup of left translation operators on $X$. The following are equivalent:
\begin{enumerate}
\item $\sup \!\set{ \frac{v(x)}{v(y)} }{ y \geq x } = \infty$.
\item For some $f \in X$, $\lim_{t \to \infty} T_tf \,\neq\, 0$.
\item $\T$ is not uniformly bounded.
\item $\T$ is not uniformly equicontinuous.
\item $\T$ is not equicontinuous.
\item $\T$ has nonzero entropy.
\item $\T$ has infinite entropy.
\end{enumerate}
\end{theorem}

\begin{proof}
We prove that $(1)$ and $(2)$ are equivalent, and then we show that $(7) \Rightarrow (6) \Rightarrow (5) \Rightarrow (4) \Rightarrow (3) \Rightarrow (1) \Rightarrow (7)$. Many of these implications are proved for $X = L^p_v(\R_+)$ and for $C_{0,v}(\R_+)$ simultaneously. Where it is necessary to distinguish between $L^p_v(\R_+)$ and $C_{0,v}(\R_+)$, we treat the case $X = L^p_v(\R_+)$ first.


To show $(2) \Rightarrow (1)$, we prove the contrapositive. Let $X = L^p_v(\R_+)$, and suppose $(1)$ does not hold, which means that there is some finite $B>0$ with
$\textstyle \sup \!\set{ \frac{v(x)}{v(y)} }{ y \geq x } = B.$
Let $f \in X$ and let $\e > 0$. Because 
$$\int_0^\infty |f(x)|^p v(x) \,dx \,=\, \norm{f}^p \,<\, \infty,$$
there is some $t_0 \geq 0$ such that
$$\int_{t_0}^\infty |f(x)|^p v(x) \,dx \,<\, \frac{\e^p}{B}.$$
This implies that for every $t \geq t_0$,
\begin{align*}
\norm{T_tf}^p &\,=\, \int_0^\infty |f(x+t)|^p v(x) \,dx \\
&\,=\, \int_0^\infty |f(x+t)|^p v(x+t)\frac{v(x)}{v(x+t)} \,dx \\
&\,\leq\, B\int_0^\infty |f(x+t)|^p v(x+t) \,dx \\
&\,=\, B\int_{t}^\infty |f(x)|^p v(x) \,dx \\
&\,\leq\, B\int_{t_0}^\infty |f(x)|^p v(x) \,dx \\
&\,<\, \e^p.
\end{align*}
Thus $\norm{T_tf} < \e$ for all $t \geq t_0$. As $f \in X$ and $\e > 0$ were arbitrary, this shows that $(2)$ does not hold.

For the second case, let $X = C_{0,v}(\R_+)$, and again suppose $(1)$ does not hold. This means there is some finite $B > 0$ with
$\textstyle \sup \!\set{ \frac{v(x)}{v(y)} }{ y \geq x } = B.$
Let $f \in X$ and let $\e > 0$. Because 
$\lim_{x \to \infty} |f(x)| v(x) = 0,$
there is some $t_0 \geq 0$ such that
$$\sup \set{|f(x)| v(x) }{ x \geq t_0 } \,<\, \frac{\e}{B}.$$
This implies that for every $t \geq t_0$,
\begin{align*}
\norm{T_tf} &\,=\, \sup \set{ |f(x+t)| v(x) }{ x \geq 0 } \\
&\,=\, \sup \set{ |f(x+t)| v(x+t) \frac{v(x)}{v(x+t)} }{ x \geq 0 } \\
&\,\leq\, B \sup \set{ |f(x+t)| v(x+t) }{ x \geq 0 } \\
&\,=\, B \sup \set{ |f(x)| v(x) }{ x \geq t } \\
&\,\leq\, B \sup \set{ |f(x)| v(x) }{ x \geq t_0 } \\
&\,<\, \e.
\end{align*}
Thus $\norm{T_tf} < \e$ for all $t \geq t_0$. As $f \in X$ and $\e > 0$ were arbitrary, this shows that $(2)$ does not hold. This completes the proof that $(2) \Rightarrow (1)$.

Note that, as in the proof of Theorem~\ref{thm:bonus}, the cases $X = L^p_v(\R_+)$ and $C_{0,v}(\R_+)$ are only superficially different: we merely had to trade our integrals for supremums. For the remaining implications, we will sometimes leave such straightforward modifications to the reader.


To show $(1) \Rightarrow (2)$, suppose $\sup \!\set{ \frac{v(x)}{v(y)} }{ y \geq x } = \infty$.
(We begin with a construction that is useful for both cases, $X = L^p_v(\R_+)$ and $X = C_{0,v}(\R_+)$.)
Recall that the admissibility condition on $v(x)$ means there exist some $M \geq 1$ and $w \in \R_+$ such that 
$v(x) \leq Me^{wt}v(x+t)$
for all $t \geq 0$. 
This implies there is some $\g > 0$ such that, for any $x,x' \in \R_+$ with $x \leq x' \leq x+\g$, $\frac{v(x)}{v(x')} \leq 2M$. (Explicitly, we may take $\g = \frac{1}{w} \ln 2$, noting that $w=0$ is impossible because having $w = 0$ would imply $\sup \!\set{ \frac{v(x)}{v(y)} }{ y \geq x } \leq M$.)

Let us define two sequences of non-negative real numbers, $\seq{y_n}{n \in \N}$ and $\seq{z_n}{n \in \N}$, via recursion such that
\begin{enumerate}
\item[$\circ$] $y_1 < z_1 < y_2 < z_2 < y_3 < z_3 < \dots$,
\item[$\circ$] $z_{n+1} > z_n + \g$ for all $n \in \N$, and
\vspace{1mm}
\item[$\circ$] $\displaystyle \frac{v(y_n)}{v(z_n)} > 2^n$ for all $n \in \N$.
\end{enumerate}

Consider the case $X=L^p_v(\R_+)$. Define a function $f: \R_+ \to \R$ as follows:
$$f(x) \,=\,
\begin{cases}
\left( \nicefrac{1}{v(z_n)2^n} \right)^{\!\nicefrac{1}{p}} &\text{ if } x \in [z_n-\g,z_n] \text{ for some } n \in \N, \\ 
0 &\text{ if not.}
\end{cases}$$
We claim that $f \in L^p_v(\R_+)$ and that $\lim_{t \to \infty} T_tf \neq 0$.

For each $n \in \N$, 
\begin{align*}
\int_{z_n-\g}^{z_n} |f(x)|^pv(x) \,dx &\,=\, \int_{z_n-\g}^{z_n} \frac{1}{v(z_n)2^n}\,v(x) \,dx \\
&\,=\, \int_{z_n-\g}^{z_n} \frac{1}{2^n}\,\frac{v(x)}{v(z_n)} \,dx \\
&\,\leq\, \int_{z_n-\g}^{z_n} \frac{1}{2^n}\,2M \,dx \\
&\,=\, \frac{M\g}{2^{n-1}},
\end{align*}
and this implies
$$\int_{0}^\infty |f(x)|^pv(x) \,dx \,=\, \sum_{n=1}^\infty \int_{z_n-\g}^{z_n} |f(x)|^pv(x) \,dx \,\leq\, \sum_{n=1}^\infty \frac{M\g}{2^{n-1}} \,<\, \infty,$$
so that $f \in L^p_v(\R_+)$ as claimed.

To show $\lim_{t \to \infty} T_tf \neq 0$, set $t_n = z_n-y_n-\g$ for each $n \in \N$. Using the admissibility of $v(x)$, observe that
$$2^n \,<\, \frac{v(y_n)}{v(z_n)} \,\leq\, Me^{w(t_n+\g)}$$
for all $n \in \N$. It follows that $\lim_{n \to \infty}t_n = \infty$. Thus in order to show $\lim_{t \to \infty} T_tf \neq 0$, it suffices to show $\lim_{n \to \infty} T_{t_n}f \neq 0$, and for this it suffices to show $\lim_{n \to \infty} \norm{T_{t_n}f}^p \neq 0$. We have
\begin{align*}
\norm{T_{t_n}f}^p &\,=\, \int_0^\infty |f(x+t_n)|^p v(x) \,dx \\
&\,\geq\, \int_{y_n}^{y_n+\g} |f(x+t_n)|^p v(x) \,dx \\
&\,=\, \int_{y_n}^{y_n+\g} \frac{1}{v(z_n)2^n} \,v(x) \,dx \\
&\,=\, \int_{y_n}^{y_n+\g} \frac{1}{2^n} \,\frac{v(x)}{v(y_n)} \,\frac{v(y_n)}{v(z_n)} \,dx \\
&\,>\, \int_{y_n}^{y_n+\g} \frac{1}{2^n} \,\frac{1}{2M} \,2^n \,dx \\
&\,=\, \frac{\g}{2M} \\
\end{align*}
for every $n \in \N$, so $\lim_{n \to \infty} \norm{T_{t_n}f}^p \neq 0$ as desired. 

Next consider the case $X = C_{0,v}(\R_+)$. The function $f$ defined in the previous case is not continuous, though functions in $C_{0,v}(\R_+)$ must be: to obtain a function suitable for this case, we simply modify the function above to make it continuous.
 
More precisely, define a function $f\colon\R_+\to\R$ as follows. For each $n \in \N$, define $f$ on $[z_n-\g,z_n]$ by letting $f(z_n)=f(z_n-\g)=0$ and $f(z_n-\frac{\g}{2}) = 1/(v(z_n)2^n)$, and then letting $f$ be linear from $z_n-\g$ to $z_n-\frac{\g}{2}$ and from $z_n-\frac{\g}{2}$ to $z_n$. If $x$ is not in $[z_n-\g,z_n]$ for any $n \in \N$, then $f(x)=0$. 

The function $f$ is continuous. (In fact, it is piecewise linear.) We claim that $f\in C_{0,v}(\R_+)$ and that $\lim_{t \to \infty} T_tf \neq 0$.

For each $n \in \N$, and every $x\in[z_n-\g,z_n]$, $|f(x)|\leq 1/(v(z_n)2^n)$, so
\begin{align*}
\sup \set{ |f(x)|v(x) }{x \in [z_n-\g,z_n]} &\,\leq\, \sup \set{ \frac{1}{v(z_n)2^n} \, v(x) }{x \in [z_n-\g,z_n]} \\
&\,=\, \sup \set{ \frac{1}{2^n} \, \frac{v(x)}{v(z_n)} }{x \in [z_n-\g,z_n]} \\
&\,\leq\, \sup \set{ \frac{1}{2^n} \, 2M }{x \in [z_n-\g,z_n]} \\
&\,=\, \frac{M}{2^{n-1}},
\end{align*}
as before, and this implies
\begin{align*}
\lim_{x \to \infty} |f(x)| v(x) &\,=\, \lim_{n \to \infty} \left( \vphantom{f^{f^f}}\hspace{-.2mm} \sup \set{ |f(x)|v(x) }{x \in [z_n-\g,z_n]} \right) \\
&\,\leq\, \lim_{n \to \infty} \frac{M}{2^{n-1}} \,=\, 0,
\end{align*}
so that $f \in C_{0,v}(\R_+)$ as claimed.

We omit the proof that $\lim_{t\to\infty}T_tf\neq 0$, as it is essentially the same as the previous case (the primary difference being that we must take supremums instead of taking integrals).

This completes the proof of $(1) \Leftrightarrow (2)$.

That $(7) \Rightarrow (6)$ is obvious. 


To show $(6) \Rightarrow (5)$, we prove the contrapositive. Suppose that $(5)$ fails: i.e., suppose that $\T$ is an equicontinuous family of functions. Let $\e > 0$, and let $K \sub X$ be compact. Note that the restriction of $\T$ to $K$ is uniformly equicontinuous. (The proof of this is essentially identical to the well-known proof that every continuous function defined on a compact metric space is uniformly continuous.) Pick $\dlt > 0$ such that for any $f,g \in K$, if $\norm{f-g}_{{}_X} < \dlt$ then $\norm{T_tf - T_tg}_{{}_X} < \e$ for all $t \in \R_+$. There is some $N \in \N$ such that $K$ can be covered by $N$ open sets of diameter $<\!\dlt$. By our choice of $\dlt$, this means that any $(t,\e)$-separated subset of $K$ has size at most $N$. As $\e > 0$ was arbitrary, it follows that $h(\T,K) = 0$. As $K$ was an arbitrary compact subset of $X$, it follows that $h(\T) = 0$.


That $(5) \Rightarrow (4)$ is obvious. 


To show $(4) \Rightarrow (3)$, we prove the contrapositive (which is just a special case of the Banach-Steinhaus Theorem.) Suppose that $(3)$ fails. Then there is some $B \geq 0$ such that $\norm{T_tf}_{{}_X} \leq B\norm{f}_{{}_X}$ for all $f \in X$ and all $t \in \R_+$. If $f,g \in X$ and $t \in \R_+$, then,
$$\norm{T_tf - T_tg}_{{}_X} \,=\, \norm{T_t(f-g)}_{{}_X} \,\leq\, B\norm{f-g}_{{}_X}.$$
It follows that for any given $\e > 0$, if $\norm{f-g}_{{}_X} < \nicefrac{\e}{B}$, then $\norm{T_tf - T_tg}_{{}_X} < \e$ for all $t \in \R_+$. Hence $\T$ is uniformly equicontinuous.


To show $(3) \Rightarrow (1)$, we again prove the contrapositive. Suppose $(1)$ does not hold, which means there is some $B > 0$ with
$\textstyle \sup \!\set{ \frac{v(x)}{v(y)} }{ y \geq x } = B.$
For the first case, suppose $X = L^p_v(\R_+)$.
If $f \in X$ and $t \in \R_+$, then,
\begin{align*}
\norm{T_tf}^p &\,=\, \int_0^\infty |f(x+t)|^p v(x) \,dx \\
&\,=\, \int_0^\infty |f(x+t)|^p v(x+t)\frac{v(x)}{v(x+t)} \,dx \\
&\,\leq\, B\int_0^\infty |f(x+t)|^p v(x+t) \,dx \\
&\,=\, B\int_{t}^\infty |f(x)|^p v(x) \,dx \\
&\,\leq\, B\int_{0}^\infty |f(x)|^p v(x) \,dx \\
&\,=\, B\norm{f}^p.
\end{align*}
It follows that $\norm{T_tf} \leq B^{\nicefrac{1}{p}}\norm{f}$. Hence $\T$ is uniformly bounded.
The second case, $X = C_{0,v}(\R_+)$, is proved similarly, by replacing integrals with supremums.


It remains to show $(1) \Rightarrow (7)$. The proof begins in a manner similar to the proof of $(1) \Rightarrow (2)$ above. Suppose $\sup \!\set{ \frac{v(x)}{v(y)} }{ y \geq x } = \infty$.
Recall that the admissibility condition on $v(x)$ means there exist some $M \geq 1$ and $w \in \R_+$ such that 
$v(x) \leq Me^{wt}v(x+t)$
for all $t \geq 0$. 
This implies there is some $\g > 0$ such that, for any $x \in \R_+$ and any $x' \in [x,x+\g]$, $\frac{v(x)}{v(x')} \leq 2M$.

Define two sequences $\seq{y_n}{n \in \N}$ and $\seq{z_n}{n \in \N}$ via recursion such that
\begin{enumerate}
\item[$\circ$] $y_1 < z_1 < y_2 < z_2 < y_3 < z_3 < \dots$,
\item[$\circ$] $z_n < z_{n+1} - \g$ for all $n \in \N$, and
\vspace{1mm}
\item[$\circ$] $\displaystyle \frac{v(y_n)}{v(z_n)} > 2^n$ for all $n \in \N$.
\end{enumerate}
For each $n \in \N$, let $J_n = [z_n-\g,z_n]$ and define $t_n = z_n-y_n-\e$.
Fix a sequence $\seq{a_n}{n \in \N}$ of positive integers such that 
$$\limsup_{n \to \infty} \frac{1}{t_n} \log a_n = \infty.$$

For each $n \in \N$ and each $1 \leq k \leq a_n$, let $J^k_n = \left[ z_n-\frac{k\g}{a_n},z_n-\frac{(k-1)\g}{a_n} \right]$, so that the $J^k_n$ form a division of $J_n$ into exactly $a_n$ adjacent closed intervals of equal width. To prove that $\T$ has infinite entropy, we consider a particular collection of functions that are zero everywhere except for on exactly one of the $J^k_n$ for each $n$.

More specifically, let us consider the set
$\C = \prod_{n = 1}^\infty \{1,2,\dots,a_n\}$ of all functions $\phi: \N \to \N$ such that $1 \leq \phi(n) \leq a_n$ for all $n \in \N$.

Consider the case $X=L^p_v(\R_+)$. For each $\phi \in \C$, define a function $f_\phi: \R_+ \to \R_+$ as follows:
$$f_\phi(x) \,=\,
\begin{cases}
\left( \nicefrac{a_n}{v(z_n)2^n} \right)^{\!\nicefrac{1}{p}} &\text{ if } x \in J^{\phi(n)}_n \text{ for some } n \in \N, \\ 
0 &\text{ if not.}
\end{cases}$$

Let us check first that $f_\phi \in L^p_v(\R_+)$ for every $\phi \in \C$.
For each $n \in \N$ and each $1 \leq k \leq a_n$, 
\begin{align*}
\int_{J_n} |f(x)|^pv(x) \,dx &\,=\, \int_{J^k_n} \frac{a_n}{v(z_n)2^n} \, v(x) \,dx \\
&\,=\, \int_{J^k_n} \frac{a_n}{2^n}\,\frac{v(x)}{v(z_n)} \,dx \\
&\,\leq\, \int_{J^k_n} \frac{a_n}{2^n}\,2M \,dx \\
&\,=\, \frac{\g}{a_n} \, \frac{a_nM}{2^{n-1}} \\
&\,=\, \frac{M\g}{2^{n-1}}.
\end{align*}
Hence
$$\int_0^\infty |f(x)|^p v(x) \,dx  \,=\, \sum_{n=1}^\infty \int_{J_n}|f(x)|^pv(x) \,dx \,\leq\, \sum_{n=1}^\infty \frac{M\g}{2^{n-1}} < \infty,$$
which means $f_\phi \in L^p_v(\R_+)$.

Let $K = \set{f_\phi}{\phi \in \C}$. We claim that $K$ is a compact subset of $L^p_v(\R_+)$, and that $h(\T,K) = \infty$. This suffices to prove $(7)$.

Note that $\C$ may be viewed as a topological space, where each set of the form $\{1,2,\dots,a_n\}$ is given the discrete topology, and the topology on $\C$ is the standard product topology. With this topology, $\C$ is compact. (In fact, it is homeomorphic to the Cantor space.) Thus, to prove that $K = \set{f_\phi}{\phi \in \C}$ is a compact subset of $X$, it suffices to show that the mapping $\phi \mapsto f_\phi$ is continuous.

To show that the mapping $\phi \mapsto f_\phi$ is continuous, let $f_\phi$ be an arbitrary point in the image of this mapping, and let $\e > 0$. 
Suppose $\psi \in \C$ and $N \in \N$, and suppose $\phi \rest \{1,2,\dots,N\} = \psi \rest \{1,2,\dots,N\}$. Then $f_\phi$ and $f_\psi$ agree on $[0,\min J_{N+1})$. This implies
\begin{align*}
\norm{f_\phi-f_\psi} &\,=\, \norm{f_\phi\chi_{{}_{[\min J_{N+1},\infty)}}-f_\psi\chi_{{}_{[\min J_{N+1},\infty)}}} \\
&\,\leq\, \norm{f_\phi\chi_{{}_{[\min J_{N+1},\infty)}}}+\norm{f_\psi\chi_{{}_{[\min J_{N+1},\infty)}}} \\
&\,=\, \left( \sum_{n = N+1}^\infty \int_{J_n} |f_\phi(x)|^p v(x) \,dx \right)^{\!\nicefrac{1}{p}} \\
&\qquad\qquad\qquad\qquad+ \left( \sum_{n = N+1}^\infty \int_{J_n} |f_\psi(x)|^p v(x) \,dx \right)^{\!\nicefrac{1}{p}} \\
&\,\leq\, \left( \sum_{n = N+1}^\infty \frac{M\g}{2^{n-1}} \right)^{\!\nicefrac{1}{p}} + \left( \sum_{n = N+1}^\infty \frac{M\g}{2^{n-1}} \right)^{\!\nicefrac{1}{p}} \\
&\,=\, 2\left( \frac{M\g}{2^N} \right)^{\!\nicefrac{1}{p}}.
\end{align*}
In particular, if $N$ is sufficiently large then 
$$\phi \rest \{1,2,\dots,N\} = \psi \rest \{1,2,\dots,N\} \quad \text{implies} \quad \norm{f_\phi-f_\psi} < \e.$$ 
But $U = \set{\psi \in \C}{\phi \rest \{1,2,\dots,N\} = \psi \rest \{1,2,\dots,N\}}$ is a basic open subset of $\C$; so we have found an open $U \sub \C$ containing $\phi$ such that every member of $U$ maps within $\e$ of $f_\phi$ in $X$. This shows that the mapping $\phi \mapsto f_\phi$ is continuous, as claimed, so $K$ is compact.

It remains to show that $h(\T,K) = \infty$. To this end, we show first that for all sufficiently small $\e$, for any $n \in \N$ there is a $(t_n,\e)$-separated subset of $K$ of size $a_n$.

Fix $\e$ with $0 < \e < \left( \frac{\g}{M} \right)^{\!\nicefrac{1}{p}}$, and let $n \in \N$. For each $i$ with $1 \leq i \leq a_n$, fix some $\phi_i \in \C$ such that $\phi(n) = i$, and let $S = \set{f_{\phi_i}}{1 \leq i \leq a_n}$. We claim that $S$ is a $(t_n,\e)$-separated subset of $K$. To see this, let $1 \leq i,j \leq a_n$ with $i \neq j$, and observe that
\begin{align*}
&\norm{T_{t_n}f_{\phi_i} - T_{t_n}f_{\phi_j}}^p \\
&\qquad\,=\, \int_0^\infty |f_{\phi_i}(x+t_n) - f_{\phi_j}(x+t_n)|^p v(x) \,dx \\
&\qquad\,\geq\, \int_{y_n}^{y_n+\g} |f_{\phi_i}(x+t_n) - f_{\phi_j}(x+t_n)|^p v(x) \,dx \\
&\qquad\,=\, \int_{y_n-\frac{i\g}{a_n}}^{y_n-\frac{(i-1)\g}{a_n}} |f_{\phi_i}(x+t_n)|^p \,v(x) \,dx \\
& \qquad \qquad \qquad + \int_{y_n-\frac{j\g}{a_n}}^{y_n-\frac{(j-1)\g}{a_n}} |f_{\phi_j}(x+t_n)|^p \,v(x) \,dx \\
&\qquad\,=\, \int_{y_n-\frac{i\g}{a_n}}^{y_n-\frac{(i-1)\g}{a_n}} \frac{a_n}{v(z_n)2^n} \,v(x) \,dx + \int_{y_n-\frac{j\g}{a_n}}^{y_n-\frac{(j-1)\g}{a_n}} \frac{a_n}{v(z_n)2^n} \,v(x) \,dx \\
&\qquad\,=\, \int_{y_n-\frac{i\g}{a_n}}^{y_n-\frac{(i-1)\g}{a_n}} \frac{a_n}{2^{n}} \,\frac{v(x)}{v(y_n)} \,\frac{v(y_n)}{v(z_n)} \,dx + \int_{y_n-\frac{j\g}{a_n}}^{y_n-\frac{(j-1)\g}{a_n}} \frac{a_n}{2^{n}} \,\frac{v(x)}{v(y_n)} \,\frac{v(y_n)}{v(z_n)} \,dx \\
&\qquad\,=\, \int_{y_n-\frac{i\g}{a_n}}^{y_n-\frac{(i-1)\g}{a_n}} \frac{a_n}{2^{n}} \,\frac{1}{2M} \,2^n \,dx + \int_{y_n-\frac{j\g}{a_n}}^{y_n-\frac{(j-1)\g}{a_n}} \frac{a_n}{2^{n}} \,\frac{1}{2M} \,2^n \,dx \\
&\qquad\,=\, \frac{\g}{a_n}\,\frac{a_n\g}{2M}+\frac{\g}{a_n}\,\frac{a_n\g}{2M} \\
&\qquad\,=\, \frac{\g}{M}, \\
\end{align*}
from which it follows that $S$ is a $(t_n,\e)$-separated subset of $K$. 

Recalling that $s_{t_n,\e}(\T,K)$ denotes the largest size of a $(t_n,\e)$-separated subset of $K$, we have $s_{t_n,\e}(\T,K) \geq |S| = a_n$. Hence
$$\limsup_{n \to \infty} \, \frac{1}{t_n} \log s_{t_n,\e}(\T,K) \,\geq\, \limsup_{n \to \infty} \, \frac{1}{t_n} a_n \,=\, \infty$$
(by our choice of the $a_n$). As in the proof of $(1) \Rightarrow (2)$,
observe that
$$2^n \,<\, \frac{v(y_n)}{v(z_n)} \,\leq\, Me^{w(t_n+\g)}$$
for all $n \in \N$. It follows that $\lim_{n \to \infty}t_n = \infty$. From this and our observation above that $\limsup_{n \to \infty} \, \frac{1}{t_n} \log s_{t_n,\e}(\T,K) \,=\, \infty$, we get
$$\limsup_{t \to \infty} \, \frac{1}{t} \log s_{t,\e}(\T,K) \,=\, \infty.$$
As this holds for all sufficiently small values of $\e$ (any $\e$ with $0 < \e < \left( \frac{\g}{M} \right)^{\!\nicefrac{1}{p}}$),
$$h(\T,K) \,=\, \lim_{\e \to 0} \, \limsup_{t \to \infty} \, \frac{1}{t} \log s_{t,\e}(\T,K) \,=\, \infty.$$
It follows that $h(\T) = \infty$, as claimed.

The case $X=C_{0,v}(\R_+)$ is very similar. We must simply be careful to define the functions $f_\phi$ so that they are continuous. This can be done as follows: for each $\phi\in\C$, we define $f_\phi$ so that on each interval $J^{\phi(n)}_n$, $f$ maps the endpoints to zero, sends the midpoint to $a_n/(v(z_n)2^n)$, and is linear in between. Then for $x$ not in any $J^{\phi(n)}_n$, let $f_\phi(x)=0$. The remainder of the proof is essentially the same as for the case $X=L^p_v(\R_+)$ (the main difference, of course, being that we must replace our integrals signs with supremums).
\end{proof}

\section{The incompleteness of the three-tiered view}

In this final section, we indicate several ways in which the three-tiered picture of $L^p_v(\R_+)$ does not completely capture the varied possibilities for the dynamics of the translation operators on $L^p_v(\R_+)$ and $C_{0,v}(\R_+)$.

Recall that a function $T$ on a space $X$ is \emph{topologically mixing} if for all nonempty open $U,V \sub X$, $T^n(U) \cap V \neq \0$ for all sufficiently large $n$. This is a strengthening of topological transitivity. In \cite{BBCP}, it was shown that the translation operators $T_t$ on $L^p_v(\R_+)$ or $C_{0,v}(\R_+)$ are all topological mixing if and only if $\lim_{x \to \infty} v(x) = 0$. Of course, this condition on $v(x)$ is strictly weaker than the integrability condition that defines our strongest tier of chaos, but strictly stronger than the condition $\liminf_{x \to \infty} v(x) = 0$ that defines the middle tier. Thus we have a type of chaotic behavior for $L^p_v(\R_+)$ that fits strictly in between the top two of our three tiers of chaos.

Looking at the top tier of chaotic behaviors for $L^p_v(\R_+)$ also highlights a difference between $L^p_v(\R_+)$ and $C_{0,v}(\R_+)$. It is fairly easy to check that $C_{0,v}(\R_+)$ contains a nonzero periodic point if $\lim_{x \to \infty} v(x) = 0$. (Indeed, if $\lim_{x \to \infty} v(x) = 0$ then $C_{0,v}(\R_+)$ contains all constant functions, which are fixed by translation.) Thus $L^p_v(\R_+)$ can exhibit the three distinct tiers of chaotic behavior in our picture, with a fourth possibility (mixing) in between the top two, but for $C_{0,v}(\R_+)$ the situation is different: at least some of the properties listed in the top tier are strictly weaker than the integrability of $v(x)$.

Another notion of chaotic behavior, introduced in \cite{Schweizer&Smital}, is distributional chaos. Let $\mu$ denote Lebesgue measure on $\R_+$. We say that $\T$ has \emph{distributional chaos} if there exists an uncountable set $S \sub X$ such that for every $f,g \in S$ with $f \neq g$, there is some $\delta > 0$ such that
	\[
	\liminf_{t\to\infty}\frac{\mu\left(\{s\in[0,t]\colon d(T_sf,T_sg)<\delta\}\right)}{t}=0
	\]
	(i.e., $f$ and $g$ are often $\dlt$-separated), and for all $\e>0$,
	\[
	\limsup_{t\to\infty}\frac{\mu\left(\{s\in[0,t]\colon d(T_sf,T_sg)<\e\}\right)}{t}=1
	\]
(i.e., $f$ and $g$ are often arbitrarily close). A single pair $f,g$ of points with this property is called a \emph{distributionally scrambled pair}.

Barrachina and Peris show in \cite{Barrachina&Peris} that $\T$ can have distributional chaos without being hypercyclic. In \cite{MGOP}, Mar\'{i}nez-Gim\'{e}nez, Oprocha, and Peris show that the backward shift operator on $\ell^p_v$ (the discrete analogue of $L^p_v(\R_+)$) can be hypercyclic and even topologically mixing, yet fail to have distributional chaos. The example they present could be adapted to show the same holds for the translation semigroup on $L^p_v(\R_+)$.
Thus the notion of distributional chaos is incomparable with our second tier of chaos, in that it neither implies the notions of chaos in that tier nor is implied by them.

Two points $f,g \in X$ form a \emph{Li-Yorke scrambled pair} if
$$\liminf_{t \to \infty} d(T_tf,T_tg) = 0 \quad \text{ and } \quad \limsup_{t \to \infty} d(T_tf,T_tg) > 0.$$
This is a weaker condition on $f$ and $g$ than the one given above; i.e., every distributionally scrambled pair is also Li-Yorke scrambled.

The fundamental observation of Schweizer an Sm\'ital in \cite{Schweizer&Smital} is that if a map $T: [0,1] \to [0,1]$ has a distributionally scrambled pair, then it has nonzero topological entropy.
Using Theorem~\ref{thm:main}, we establish an even stronger result for translations on $L^p_v(\R_+)$ and $C_{0,v}(\R_+)$.

\begin{theorem}
Let $X$ denote one of the Banach spaces $L^p_v(\R_+)$ or $C_{0,v}(\R_+)$, where $v$ is an admissible weight function, and let $\T = \set{T_t}{t \in \R_+}$ denote the semigroup of left translation operators on $X$. 
If there exists a Li-Yorke scrambled pair for $\T$, then $h(\T)=\infty$.
\end{theorem}
\begin{proof}
Suppose there exist some $f,g \in X$ that form a Li-Yorke scrambled pair for $\T$.
Then letting $h = f-g$, we have
$$\limsup_{t \to \infty} \norm{T_th}_{{}_X} \,=\, \limsup_{t \to \infty} d(T_tf,T_tg) \,>\, 0.$$
which implies that $T_th$ does not converge to $0$ as $t \to \infty$. Thus by Theorem~\ref{thm:main}, $h(\T)=\infty$.
\end{proof}

Finally, let us look at what happens below the lowest tier of chaos included in our picture from Section 2, that is, translations on $L^p_v(\R_+)$ and $C_{0,v}(\R_+)$ where $\sup \set{\frac{v(x)}{v(y)}}{x \leq y} = b$ for some $b > 0$. These dynamical systems have zero entropy, and every point tends to $0$ under iteration. One might be tempted to think they are all helplessly tame, and can exhibit no dynamically interesting behavior. We consider two examples, and show that they can in fact behave rather differently.

Suppose $X$ is a metric space and $T: X \to X$ is a mapping. Given $\e > 0$, a sequence $\seq{f_i}{0 \leq i \leq n}$ of points in $X$ is called an \emph{$\e$-chain} from $f_0$ to $f_n$ if $d(T(f_i),f_{i+1}) < \e$ for every $i < n$. The idea is that an $\e$-chain is a finite piece of the orbit of $f_0$, but computed with a small error at every step, an error of size less than $\e$. The map $T$ is called \emph{chain transitive} if for any $f,g \in X$ and any $\e > 0$, there is an $\e$-chain from $f$ to $g$.

It is fairly easy to check that every transitive dynamical system is also chain transitive. Thus, for $L^p_v(\R_+)$, in our top two tiers of chaos every $T \in \T$ is chain transitive. We show now that chain transitivity may or may not hold in the non-chaotic zone beneath the bottom tier.

\begin{example}
Suppose $v$ is a constant function, $v(x) = c$. Then we claim that every $T \in \T \setmins \{T_0\}$ is chain transitive for $X = L^p_v(\R_+)$. (The example can be modified to show the same for $X = C_{0,v}(\R_+)$, but we leave the details of this to the reader.) Fix $T = T_t$ with $t > 0$. To prove $T$ is chain transitive, we begin by showing that for every $g \in X$ and $\e > 0$, there is an $\e$-chain from $0$ to $g$.
So let $\e > 0$ and let $g \in X$, and fix $n \in \N$ larger than $\nicefrac{\norm{g}}{\e}$. 
For each $i \leq n$, let 
$$g_i(x) \,=\, \frac{i}{n} \, T^{-(n-i)}g(x) \,=\, \begin{cases}
0 & \text{ if } x < (n-i)t, \\
\frac{i}{n} \, g(x-(n-i)t) & \text{ if } x \geq (n-i)t.
\end{cases}$$
(In other words, $g_i$ is a copy of $g$ that has been scaled down by a factor of $\nicefrac{i}{\e}$, and then shifted to the right by $(n-i)t$ units.)
We claim that $\seq{g_i}{0 \leq i \leq n}$ is the required $\e$-chain from $0$ to $g$. It is clear that $g_0 = 0$ and that $g_n = g$. For each $i \leq n$, we have
$$Tg_i \,=\, \textstyle \frac{i}{i+1} \, g_{i+1}.$$ 
Furthermore, $\norm{T_s^{-1}g} = \norm{g}$ for all $s \in \R_+$ (because $v(x)$ is constant), and it follows that 
$$\textstyle \norm{g_i} = \norm{\frac{i}{n}T^{-(n-i)}g} = \frac{i}{n}\norm{T^{-(n-i)}g} = \frac{i}{n}\norm{g}$$ 
for every $i$. 
Hence
$$\norm{Tg_i - g_{i+1}} \,=\, \norm{\textstyle \frac{i}{i+1} \, g_{i+1} - g_{i+1}} \,=\, \frac{1}{i+1}\norm{g_{i+1}} \,=\, \frac{1}{i+1}\frac{i+1}{n}\norm{g} \,=\, \frac{1}{n}\norm{g} \,<\, \e.$$
Thus $\seq{g_i}{0 \leq i \leq n}$ is an $\e$-chain from $0$ to $g$, as claimed. Next note that for any $f \in X$ and any $\e > 0$, there is an $\e$-chain from $f$ to $0$. The easiest way to see this is to observe that $\lim_{t \to \infty}T_tf = 0$ by Theorem~\ref{thm:main}, so there is some $n \in \N$ such that $\norm{T^n f} < \e$, in which case
$$\langle f,Tf,T^2f,\dots,T^{n-1}f,0 \rangle$$
is an $\e$-chain from $f$ to $0$. Finally, if $f,g \in X$ and $\e > 0$, then we may obtain an $\e$-chain from $f$ to $g$ by concatenating an $\e$-chain from $f$ to $0$ with an $\e$-chain from $0$ to $g$. Hence $T$ is chain transitive, as claimed.
\end{example}

\begin{example}
Suppose $v(x) = c^x$ for some $c > 1$. Then we claim that no $T \in \T \setmins \{T_0\}$ is chain transitive for $X = L^p_v(\R_+)$. (We leave it to the reader to show that a similar argument proves the same thing for $X = C_{0,v}(\R_+)$.) Fix $T = T_t$ with $t > 0$. To prove $T$ is not chain transitive, first observe that for any $h \in X$,
\begin{align*}
\norm{Th}^p &\,=\, \int_0^\infty |h(x+t)|^p c^x \,dx \\
&\,=\, c^{-t} \int_0^\infty |h(x+t)|^p c^{x+t} \,dx \\
&\,=\, c^{-t} \int_t^\infty |h(x)|^p c^x \,dx \\
&\,=\, c^{-t} \norm{h}^p.
\end{align*}
Thus, for any $f,g \in X$,
$$\norm{Tf-Tg} \,=\, \norm{T(f-g)} \,=\, c^{-\nicefrac{t}{p}}\norm{f-g}.$$
(This shows that $T$ is a contraction mapping.) Let $f \in X \setmins \{0\}$. We claim that for all sufficiently small $\e > 0$, there is no $\e$-chain from $0$ to $f$. Let $0 < \e < (\frac{1}{2}-\frac{1}{2}c^{-\nicefrac{t}{p}})\norm{f}$. (Note that $0 < c^{-\nicefrac{t}{p}} < 1$, because $t$ and $p$ are both positive.)
If $\norm{g} < \frac{1}{2}\norm{f}$, then 
$$\textstyle \norm{Tg} \,=\, c^{-\nicefrac{t}{p}}\norm{g} \,<\, \frac{1}{2}c^{-\nicefrac{t}{p}}\norm{f} \,<\, \frac{1}{2}\norm{f} - \e.$$
This implies that if $\seq{g_i}{0 \leq i \leq n}$ is an $\e$-chain, then for any $i < n$, 
$$\textstyle \norm{g_i} < \frac{1}{2}\norm{f} \ \ \text{ implies } \ \ \norm{g_{i+1}} < \frac{1}{2}\norm{f}.$$
Thus any $\e$-chain beginning in the open set $B(0,\frac{1}{2}\norm{f})$ must remain in that open set. In particular, there is no $\e$-chain from $0$ to $f$.
\end{example}

\end{document}